\documentclass[a4paper,12pt]{article}
\usepackage{graphicx,amsmath,latexsym,amssymb,amsthm}
\newtheorem{thm}{Theorem}[section]

\newtheorem{defn}[thm]{Definition}
\newtheorem{rem}[thm]{Remark}

\newtheorem{exm}[thm]{Example}
\usepackage[colorlinks,urlcolor=blue]{hyperref}

\usepackage{enumerate}
\date{~}
\begin{document}

\title{Semi Compactness in Multiset Topology}
\author{J. Mahanta*, D. Das$^\perp$ \\
Department of Mathematics\\
         NIT Silchar, Assam, 788 010, India.\\
         *mahanta.juthika@gmail.com, $^\perp$ deboritadas1988@gmail.com}

\maketitle

\begin{abstract}
In this paper, we introduce and study the concepts of semi open (SOM) and semi closed (SCM) M-sets in multiset topological spaces. With this generalization of the notions of open and closed sets in M-topology, we generalize the concept of compactness in M-topology as semi compactness. Further semi compactness is generalized as semi whole compactness, semi partial whole compactness and semi full compactness. Some characterizations of these compact spaces are studied in the setting of multiset theory. In each step, several remarks with proper justifications are provided taking the well existing theories of general topology as the base of our study.
\end{abstract}

\vspace{.3 cm}

\section{Introduction}

Classic set theory is a basic concept to represent various situations in Mathematical notation where repeated occurrences of elements are not allowed. But in various circumstances repetition of elements become mandatory to the system; eg., considering a graph with loops, in chemical bonding, molecules of a substance etc. Taking this fact into consideration, Blizard \cite{BW} first studied multiset in 1989. In the same sense, Dedekind \cite{DR} stated that each element in the range of a function can be thought of as having a multiplicity equal to the number of elements in the domain that are mapped to it. The theory of multisets have been studied by many other authors in different senses \cite{GK}, \cite{HM}, \cite{LA}, \cite{SD}, \cite{SDA}, \cite{SS} and \cite{WH}.\\

A wide application of M-sets can be found in various branches of Mathematics. Algebraic structures for multiset space have been constructed by Ibrahim {\it{et al}}. in \cite{IA}. In \cite{OF}, use of multisets in colorings of graphs have been discussed by F. Okamota {\it{et al}}.  Application of M-set theory in decision making can be seen in \cite{YY}. A. Syropoulos \cite{SA}, presented a categorical approach to multisets along with partially ordered multisets. V. Venkateswaran \cite{VV} found a large new class of multisets Wilf equivalent pairs which is claimed to be the most general multiset Wilf equivalence result to date. In 2012, Girish and John \cite{GJ} introduced multiset topologies induced by multiset relations. The same authors further studied the notions of open sets, closed sets, basis, sub basis, closure and interior, continuity and related properties in M-topological spaces in \cite{GKJ}. \\

In this paper, we introduce the concept of semi open and semi closed sets in multiset theory. In classical set theory, semi open and semi closed sets were first studied by N. Levine \cite{LN}. Since its introduction, semi closed and semi open sets have been studied by different authors [ \cite{DC}, \cite{DC2}, \cite{GM}, \cite{TT} etc].\\

This paper begins with the initiation of semi open M-sets (SOM-sets) and semi closed M-sets (SCM-sets) in M-topology. Then we focus on the study of various set theoretic properties of SOM-sets and SCM-sets. Further we introduce the concept of semi compactness in M-topological space along with certain characterizations.

\section{Preliminaries}
Below are some definitions and results as discussed in \cite{GJ}, which are required in the sequel.

\begin{defn} 
An M-set $M$ drawn from the set $X$ is represented by a function Count $M$ or $C_M : X \longrightarrow W$, where $W$ represents the set of whole numbers.\\

Here $C_M (x)$ is the number of occurrences of the element $x$ in the M-set $M$. We represent the M-set $M$ drawn from the set $X = \{x_1, . . . , x_n\}$ as $M= \{m_1/x_1, m_2/x_2, . . . , m_n/x_n\}$ where $m_i$ is the number of occurrences of the element $x_i, i=1, 2, . . ., n$ in the M-set $M$. Those elements which are not included in the M-set have zero count. 
\end{defn}
Note: Since the count of each element in an $M$-set is always a non-negative integer so we have taken $W$ as the range space instead of $N$.
\begin{exm}
Let $X = \{a, b, c\}$. Then $M = \{3/a, 5/b, 1/c\}$ represents an M-set drawn from $X$.
\end{exm}

Various operations on M-sets are defined as follows:

If $M$ and $N$ are two M-sets drawn from the set $X$, then
\begin{itemize} 
\item $M=N \Leftrightarrow C_M (x) = C_ N (x)~ \forall  x \in X$.
\item $M \subseteq N \Leftrightarrow C_M (x) \leq  C_ N (x) ~ \forall  x \in X$.
\item $P = M \cup N \Leftrightarrow C_P (x) = max \{C_ M (x), C_ N (x)\} ~ \forall x \in X$.
\item $P = M \cap N \Leftrightarrow C_P (x) = min \{C_ M (x), C_ N (x)\} ~ \forall x \in X$.
\item $P = M \oplus N \Leftrightarrow C_P (x) = C_ M (x) + C_ N (x) ~ \forall x \in X$.
\item $P = M \ominus N \Leftrightarrow C_P (x) = max \{C_ M (x)- C_ N (x), 0\} ~ \forall x \in X$, where $\oplus$ and $\ominus$ represents M-set addition and M-set subtraction respectively. 
\end{itemize}
\textbf{Operations under collections of M-sets:} Let $[X]^w$ be an M-space and $\{M_i ~|~i \in I \}$ be a collection of M-sets drawn from $[X]^w$. Then the following operations are defined
\begin{itemize}
\item $\underset{i \in I} {\bigcup} M_i = \{C_{M_i}(x)/x ~ |~ C_{M_i}(x) = max \{C_{M_i}(x) ~|~ x \in X\}\}$.
\item $\underset{i \in I} {\bigcap} M_i = \{C_{M_i}(x)/x ~ |~ C_{M_i}(x) = min \{C_{M_i}(x) ~|~ x \in X\}\}$.
\end{itemize}

\begin{defn}
The support set of an M-set $M$, denoted by $M^*$ is a subset of $X$ and is defined as $M^* = \{x \in X ~|~ C_M(x) > 0\}$. $M^*$ is also called root set.
\end{defn}

\begin{defn}
An M-set $M$ is called an empty M-set if $C_M(x) = 0, ~~ \forall  ~ x \in X$.
\end{defn}

\begin{defn}
A domain $X$, is defined as the set of elements from which M-set are constructed. The M-set space $[X]^w$ is the set of all M-sets whose elements are from X such that no element occurs more than w times.
\end{defn}

\begin{rem}
It is clear that the definition of the operation of M-set addition is not valid in the context of M-set space $[X]^w$, hence it was refined as\\
$C_{M_1 \oplus M_2}(x) = min \{w, C_{M_1}(x) + C_{M_2}(x)\}$ for all $x \in X$.
\end{rem}

 In multisets the number of occurrences of each element is allowed to be more than one which leads to generalization of the definition of subsets in classical set theory. So, in contrast to classical set theory, there are different types of subsets in multiset theory.

\begin{defn}
A subM-set N of M is said to be a whole subM-set if and only if $C_{N}(x)=C_{M}(x)$ for every $x \in N$.
\end{defn}

\begin{defn}
A subM-set N of M is said to be a partial whole subM-set if and only if $C_{N}(x)=C_{M}(x)$ for some $x \in N$.
\end{defn}

\begin{defn}
A subM-set N of M is said to be a full subM-set if and only if $C_{N}(x) \leq C_{M}(x)$ for every $x \in N$.
\end{defn}

As various subset relations exist in multiset theory, the concept of power M-set can also be generalized as follows:
\begin{defn} Let $M \in [X]^w$ be an M-set.
\begin{itemize}
\item The power M-set of M denoted by ${\mathcal{P}}(M)$ is defined as the set of all subM-sets of M.
\item The power whole M-set of M denoted by ${\mathcal{PW}}(M)$ is defined as the set of all whole subM-sets of M.
\item The power full M-set of M denoted by ${\mathcal{PF}}(M)$ is defined as the set of all full subM-sets of M.
\end{itemize}
\end{defn}
The power set of an M-set is the support set of the power M-set and is denoted by ${\mathcal{P}}^{*}(M)$.
\begin{defn}
Let $M \in [X]^w$ and $\tau \subseteq {\mathcal{P}}^{*}(M) $. Then $\tau$ is called an M-topology if it satisfies the following properties:
\begin{itemize}
\item The M-set M and the empty M-set $\phi$ are in $\tau$.
\item The M-set union of the elements of any subcollection of $\tau$ is in $\tau$.
\item The M-set intersection of the elements of any finite subcollection of $\tau$ is in $\tau$.
\end{itemize}
The elements of $\tau $ are called open M-set. A subM-set $N$ of a M-topological space $M$ is said to be closed if the M-set $M \ominus N$ is open.
\end{defn}

\begin{defn}
Given a subM-set A of an M-topological space M in $[X]^w$
\begin{itemize}
\item The interior of A is defined as the M-set union of all open M-sets contained in A and is denoted by $int(A)$ i.e.,
$C_{int(A)}(x)= C_{\cup G}(x)$ where G is an open M-set and $G \subseteq A$.
\item The closure of A is defined as the M-set intersection of all closed M-sets containing A and is denoted by $cl(A)$ i.e.,
$C_{cl(A)}(x)= C_{\cap K}(x)$ where G is a closed M-set and $A \subseteq K$.
\end{itemize}
\end{defn}

\begin{defn} If $M$ is an M-set, then the M-basis for an M-topology in $[X]^w$ is a collection $\mathcal{B}$ of subM-sets of $M$ such that 
\begin{itemize}
\item For each $x {\in}^m M$, for some $m>0$ there is at least one M-basis element $B \in \mathcal{B}$ containing $m/x$.
\item If $m/x$ belongs to the intersection of two M -basis elements $P$ and $Q$, then $\exists$ an M- basis element $R$ containing $m/x$ such that $R \subseteq P \cap Q$ with $C_R(x)= C_{P \cap Q}(x)$ and $C_R(y) \leq  C_{P \cap Q}(y)$ $\forall y \neq x$.
\end{itemize}
\end{defn}

\begin{defn}
Let $(M, \tau)$ be an M-topological space and N is a subM-set of M. The collection $\tau_N = \{N \cap U : U \in \tau\}$ is an M-topology on N, called the subspace M-topology. With this M-topology, N is called a subspace of M.
\end{defn}

\section{Semi Open M-sets and Semi Closed M-sets}

In this section, we introduce the notions of semi open and semi closed sets in multiset theory, viz., semi open M-set (SOM-set) and semi closed M-set (SCM-set). The topological structures of a space can be characterized by the notions of open and closed sets. With the generalization of open and closed sets in M-topology, we get a broader class of sets and hence can study the space to get some interesting results which will be more applicable in certain areas of study. 
\\Throughout the paper we shall employ the following definition for complement of an M-set:
\begin{defn}
The M-complement of a subM-set $N$ in an M-topological space $(M, \tau)$ is denoted and defined as $N^{c}= M \ominus N$.
\end{defn}

\begin{defn}
An M-set S in an M-topology $(M, \tau)$ is said to be semi open M-set(SOM-set) iff $\exists$ an open M-set O such that 
\begin{center}
 $C_O(x) \leq C_S(x) \leq C_{cl(O)}(x)$ for all $x \in X$.
\end{center}
\end{defn}

\begin{exm}
Let $X = \{a, b, c\}$ and $M = \{5/a, 2/b, 3/c\}$ be an M-set in  $[X]^5$. Then $\tau = \{M, \phi, \{5/a, 2/b\}, \{3/c\}, \{1/a, 2/b\}, \{1/a, 2/b, 3/c\}\}$ is an M-topology on M. \\
In this M-topology the possible SOM-sets are $\phi, M, \{5/a, 2/b\}, \{3/c\}, \{1/a, 2/b\},\\ \{1/a, 2/b, 3/c\}, \{2/a, 2/b\}, \{3/a, 2/b\}, \{4/a, 2/b\}, \{2/a, 2/b, 3/c\},\{3/a, 2/b, 3/c\}, \\ \{4/a, 2/b, 3/c\}$.
\end{exm}

\begin{defn}
An M-set T in an M-topology $(M, \tau)$ is said to be semi closed M-set(SCM-set) iff $\exists$ a closed M-set P such that 
\begin{center}
 $C_{int(P)}(x) \leq C_T(x) \leq C_P(x)$ for all $x \in X$.
\end{center}
The M-complement of SOM-sets are SCM-sets
\end{defn}

\begin{defn}
Let $(M,\tau)$ be an M-topology. Then 
\begin{itemize}
\item the semi closure of an M-set $A$ is denoted by $scl(A)$ and defined as $C_{scl(A)}(x)= C_{\cap  T}(x)$ where $T$ is a SCM-set containing $A$.
\item the semi interior of an M-set $A$ is denoted by $sint(A)$ and defined as $C_{sint(A)}(x)= C_{\cup  S}(x)$ where $S$ is a SOM-set contained in $A$.
\end{itemize}
\end{defn}

\begin{thm}
Let $ \{S_{\lambda} : \lambda \in \Lambda \}$ be a collection of SOM-sets in an M-topology $(M, \tau)$. Then arbitrary union of $S_{\lambda}$ is also SOM-set.
\end{thm}

\begin{proof}
Since each $S_ \lambda$ is SOM-set so $\exists$ an open M-set $O_ \lambda$ for each $\lambda$ such that \\
 $C_{O_{\lambda}}(x) \leq C_{S_{\lambda}}(x) \leq C_{cl(O_{\lambda})}(x)$ for all $x \in X$ and $\lambda \in \Lambda $.
 Now, taking arbitrary union over $\lambda$,\\
 \begin{center} 
 $C_{\underset{\lambda}{\cup}O_{\lambda}}(x) \leq C_{\underset{\lambda}{\cup}S_{\lambda}}(x) \leq C_{\underset{\lambda}{\cup}cl(O_{\lambda})}(x)$ for all $x \in X$. \end{center}
 This imply $\underset{\lambda}{\cup}S_{\lambda}$ is a SOM-set, because $\underset{\lambda}{\cup}O_{\lambda}$ is an open M-set being the arbitrary union of open sets.
\end{proof}

\begin{rem}
Intersection of two SOM-sets may not be a SOM-set.
\end{rem}

\begin{rem}
The collection of all SOM-sets doesn't form an M-topology since intersection of two SOM-sets may not be a SOM-set. However, this collection forms an M-topology iff the closure of all open M-sets are open or equivalently disjoint open sets have disjoint closures in $(M, \tau)$.
\end{rem}

\begin{thm}
The union of a SOM-set with an open M-set is again a SOM-set.
\end{thm}
\begin{proof}
Let $O$ be an open M-set and $S$ be a SOM-set. So $\exists$ an open M-set $A$ such that $C_A(x) \leq C_S(x) \leq C_{cl(A)}(x) \Rightarrow C_{A \cup O}(x) \leq C_{S \cup O}(x) \leq C_{cl(A) \cup O}(x)$. Now, we have $C_{cl(A) \cup O} \leq C_{cl(A) \cup cl(O)}$ = $C_{cl(A \cup O)}$. Hence, $ C_{A \cup O}(x) \leq C_{S \cup O}(x) \leq C_{cl(A \cup O)}(x) \Rightarrow S \cup O$ is a SOM-set.
\end{proof}

\begin{thm}
Let $ \{T_{\lambda} : \lambda \in \Lambda \}$ be a collection of SCM-sets in an M-topology $(M, \tau)$. Then arbitrary intersection of $T_{\lambda}$ is also SCM-set.
\end{thm}

\begin{rem}
Union of two SCM-sets may not be a SCM-set.
\end{rem}

\begin{thm}
If a SOM-set S is such that  $C_S(x) \leq C_N(x) \leq C_{cl(S)}(x)$ for all $x \in X$, then the M-set N is also a SOM-set.
\end{thm}
\begin{proof}
As S is a SOM-set $\exists$ an open M-set O such that  $C_O(x) \leq C_S(x) \leq C_{cl(O)}(x)$ for all $x \in X$. Then by hypothesis  $C_O(x) \leq C_N(x)$ for all $x \in X$ and  $C_{cl(S)}(x) \leq C_{cl(O)}(x) $ for all $x \in X$. This implies $C_N(x) \leq C_{cl(S)}(x) \leq C_{cl(O)}(x) \Rightarrow C_O(x) \leq C_N(x) \leq C_{cl(O)}(x)$. Hence, N is a SOM-set.
\end{proof}

\begin{thm}
If a SCM-set T is such that  $C_{int(T)}(x) \leq C_R(x) \leq C_T(x)$ for all $x \in X$, then the M-set R is also a SCM-set.
\end{thm}

\begin{thm}
The following conditions are equivalent:
\begin{enumerate}[(i)]
\item S is SOM-set;
\item $C_S(x) \leq C_{cl(int(S))}(x)$;
\item $C_{int(cl(S^c))}(x) \leq C_{S^c}(x)$;
\item $S^c$ is SCM-set.
\end{enumerate}
\end{thm}

\begin{proof}
$(i) \Rightarrow (ii)$ Let $S$ be a SOM-set. So $\exists$ an open M-set $A$ such that  $C_A(x) \leq C_S(x) \leq C_{cl(A)}(x)$ for all $x \in X$.
 Since $A$ is an open M-set $C_S(x) \leq C_{cl(intA)}(x)$. Now  $C_A(x) \leq C_S(x) \Rightarrow  C_{(cl(intA))}(x) \leq C_{(cl(int(S)))}(x) $. Thus we have $C_S(x) \leq C_{cl(int(S))}(x)$.\\
 $(ii) \Rightarrow (iii)$ Taking complement of $(ii)$, $C_{int(cl(S^c))}(x) \leq C_{S^c}(x)$. \\
 $(iii) \Rightarrow (iv)$ $cl(S^c)$ is a closed M-set such that $C_{int(cl(S^c))}(x) \leq C_{S^c}(x) \leq C_{cl(S^c)}(x)$. So $S^c$ is a SCM-set.\\
 $(iv) \Rightarrow (i)$ Since $S^c$ is a SCM-set, $\exists$ a closed M-set $B$ such that $C_{int(B)}(x) \leq C_{S^c}(x) \leq C_{B}(x) \Rightarrow C_{cl(B^c)}(x) \geq C_{S}(x) \geq C_{B^c}(x) \Rightarrow S$ is a SOM-set. 
\end{proof}

On the basis of above discussion we list the conditions for an M-set $A$ to be SOM-set as follows:
\begin{itemize}
\item If $A$ is open;
\item If $A$ is clopen;
\item If $A$ is the closure of some open M-set;
\item If $A$ is the interior of some M-set;
\item If $C_A(x) \leq C_{cl(intA)}(x)$ $\forall x \in X$;
\item If $C_S(x) \leq C_A(x) \leq C_{cl(S)}(x)$ for all $x \in X$, for some SOM-set $S$.
\end{itemize}

Similarly, the conditions for an M-set $B$ to be SCM-set are as follows:
\begin{itemize}
\item If $B$ is closed;
\item If $B$ is clopen;
\item If $B$ is the closure of some M-set;
\item If $B$ is the interior of some closed M-set;
\item If $ C_{int(clB)}(x) \leq C_B(x)$ $\forall x \in X$;
\item If $C_{int(T)}(x) \leq C_B(x) \leq C_T(x)$ for all $x \in X$, for some SCM-set $T$.
\end{itemize}

\begin{rem}
A SOM(SCM)-set need not be open(closed). It is open(closed) provided that the topology is either the discrete topology or Power set of whole subM-sets of M.
\end{rem}

\section{Semi Compactness in M-Topology}
 In this section, we introduce the concept of semi compactness in M-topology. The notion of compactness in an topological space is based on the open sets of the topology. So with the generalization of open sets as SOM-sets, the notion of compactness have been generalized as Semi-compactness in the present article.

\begin{defn} A collection $\{N_{\lambda}: \lambda \in \Lambda \}$ of M-sets is said to be a cover of an M-topological space $(M,\tau)$ if
\begin{center} $C_M(x) \leq C_{\underset{\lambda}{\cup}N_{\lambda}}(x)$ $\forall x \in X$.\end{center}
 Then we say $M$ is covered by $\{N_{\lambda}: \lambda \in \Lambda \}$.
\end{defn}

\begin{defn} If each $N_{\lambda}$ is SOM-set, then the open cover is said to be a semi open cover.
\end{defn}

\begin{defn} Any subcollection of a semi open cover is said to be semi subcover if it covers M.
\end{defn}

\begin{defn} Any subcollection of a semi open cover where each element is a whole subM-set is said to be semi whole subcover if it covers M.
\end{defn}

\begin{defn} An M-topological space $(M,\tau)$ is said to be a \textbf{semi compact space} if every semi open cover of M has a finite semi open subcover i.e. for any collection $\{N_{\lambda}: \lambda \in \Lambda \}$ of SOM-sets covering $M$, $\exists$ a finite subcollection $\{N_{\lambda_i}:i=1,2,....,n\}$ such that $C_M(x) \leq C_{\overset{n}{\underset{i=1}{\cup}}N_{\lambda_i}}(x)$, $\forall x \in X$.
\end{defn}

\begin{exm} 
Let $X = Z_+$, the set of positive integers, 

\begin{math} M = \begin{cases}2n/n  & if ~ n ~is~ even; \\ (2n+1)/n & n ~is ~odd. \end{cases} 
\end{math}
\\Then $\tau = \{M, \phi, \{4/2, 8/4, 12/6 , . . . \}, \{3/1, 7/3, 11/5, . . . \} \}$ is a M-topology. Now any semi open cover of $M$ must contain either $M$ or $\{4/2, 8/4, 12/6 , . . . \}$ and $\{3/1, 7/3, 11/5, . . . \}$ or all of them. In any case, the cover is finite and hence $M$ is semicompact.
\end{exm}

There are various types of subset relations in M-set theory. Instead of taking only M-subset relation, if we consider the whole/partial whole or full subset relations the notion of semi compactness can be generalized as follows:
\begin{defn} An M-topological space $(M,\tau)$ is said to be a \textbf{semi-whole compact space} if every semi open cover of M has a finite semi-whole subcover.\end{defn}

\begin{defn} An M-topological space $(M,\tau)$ is said to be a \textbf{semi-partial whole compact space} if every semi open cover of M has a finite semi- partial whole subcover.\end{defn}

\begin{defn} An M-topological space $(M,\tau)$ is said to be a \textbf{semi-full compact space} if every semi open cover of M has a finite semi-full subcover.\end{defn}

\begin{defn} An arbitrary collection ${\mathcal{C}}= \{O_1, O_2,...........\}$ of M-sets is said to have Finite intersection property(FIP) if intersection of elements of every finite subcollection $\{O_1, O_2,.........,O_n\}$ of $\mathcal{C}$ is non-empty, \\
i.e., $C_{\overset{n}{\underset{i=1}{\cap }} O_i}(x) \neq C_{\phi}(x)$.
\end{defn}

\begin{thm}
Let $(M, \tau)$ be an M-topology. $M$ is semi compact iff every collection $\mathbf{C} = \{ T_{\lambda} : \lambda \in \Lambda \}$ of SCM-sets in $M$ having the FIP is such that $C_{\underset {\lambda \in \Lambda}{\cap } T_\lambda}(x) \neq C_{\phi}(x)$.
\end{thm}

\begin{proof}

Let $M$ be semi compact and $\{ T_\lambda : \lambda \in \Lambda \}$ be collection of SCM-sets with FIP such that $C_{\underset {\lambda \in \Lambda}{\cap } T_\lambda}(x) = C_{\phi}(x) \Rightarrow C_{\underset {\lambda \in \Lambda}{\cup } T_{\lambda}^c}(x) = C_{M}(x) \Rightarrow \{ T_\lambda^c : \lambda \in \Lambda \}$ forms semi open cover of $M$. So $\exists, \{ T_{\lambda_i}^c : i = 1, 2, . . . , n \}$ such that $C_{\overset{n}{\underset{i=1}{\cup }} T_{\lambda_i}^c}(x) = C_{M}(x) \Rightarrow C_{\overset{n}{\underset{i=1}{\cap }} T_{\lambda_i}}(x) = C_{\phi}(x)$ for all $x \in X$, which is a contradiction.\\
Conversely, let every collection of SCM-sets having FIP be such that $C_{\underset {\lambda \in \Lambda}{\cap } T_\lambda}(x) \neq C_{\phi}(x)$. If possible let $M$ be not semi compact. So $\exists$ a semi open cover $\{ S_\lambda : \lambda \in \Lambda \}$ which has no finite subcover. Hence $C_M(x) > C_{\overset{n}{\underset{i=1}{\cup }} S_{\lambda_i}}(x) \Rightarrow C_{\phi}(x) \leq C_{\overset{n}{\underset{i=1}{\cap }} S_{\lambda_i}^c}(x)$, which is a contradiction.   
\end{proof}

\begin{thm}
Let $(M, \tau)$ be an M-topology. $M$ is semi compact iff every collection $\mathbf{C} = \{ N_{\lambda} : \lambda \in \Lambda \}$ of M-sets in $M$ having the FIP is such that $C_{\underset {\lambda \in \Lambda}{\cap } scl(N_\lambda)}(x) \neq C_{\phi}(x)$.
\end{thm}

\begin{proof}
Let $(M, \tau)$ be semi compact. If possible let $\mathbf{C} = \{ N_{\lambda} : \lambda \in \Lambda \}$ be a collection of M-sets in $M$ having the FIP be such that $C_{\underset {\lambda \in \Lambda}{\cap } scl(N_\lambda)}(x) = C_{\phi}(x) \Rightarrow C_{\underset {\lambda \in \Lambda}{\cup } (scl(N_\lambda))^c}(x) = C_{M}(x)$. Therefore $\{(scl(N_\lambda))^c : \lambda \in \Lambda\}$ forms a semi open cover of $M$. By semi compactness of $M$, $\exists$ a finite subcover $\{(scl(N_{\lambda_i}))^c : i = 1, 2, . . . , n \}$ such that $C_{\overset{n}{\underset{i=1}{\cup }} (scl(N_{\lambda_i}))^c}(x) = C_{M}(x) \Rightarrow C_{\overset{n}{\underset{i=1}{\cup }} N_{\lambda_i}}(x) \leq C_{\phi}(x)$, a contradiction.\\
Conversely, let the hypothesis and assume that $M$ is not compact. Then  $\exists$ a semi open cover $\{ S_\lambda : \lambda \in \Lambda \}$ which has no finite subcover. So $\exists$ a finite subcollection  $\{ S_{\lambda_i} : i=1,2,...,n \}$ such that, $C_{\overset{n}{\underset{i=1}{\cup }} S_{\lambda_i}}(x) < C_{M}(x) \Rightarrow C_{\overset{n}{\underset{i=1}{\cap }} S_{\lambda_i}^c}(x) \geq C_{\phi}(x)$. Therefore $\{S_{\lambda}^c : \lambda \in \Lambda\}$ is a family of M-sets with FIP. Now
$C_{\underset {\lambda \in \Lambda}{\cup } S_\lambda}(x) \geq C_{M}(x) \Rightarrow C_{\underset {\lambda \in \Lambda}{\cap } S_\lambda^c}(x) = C_{\phi}(x) \Rightarrow  C_{\underset {\lambda \in \Lambda}{\cap } (scl(S_{\lambda_i}))^c}(x) = C_{\phi}(x)$, a contradiction.
\end{proof}

\begin{rem}
The Theorems 4.11 and 4.12 hold good for semi whole/partial whole and full compact spaces.
\end{rem}

\begin{rem}
In general topology, a topological space is compact if every basic open cover (subbasic open cover) has a finite subcover. Considering M-topology, it is not necessarily true in context of semi open sets. Because, intersection of two SOM-sets is not necessarily a SOM-set and hence a collection of SOM-sets may not form a M-topology. 
\end{rem}

\begin{thm}
Let $(M, \tau)$ be an M-topology and $(N, \tau_N)$ be its subspace. Let $A$ be an M-set such that $C_A (x) \leq C_N (x) \leq C_M (x)$. Then A is $\tau$-semi compact iff A is $\tau_N$ semi compact.
\end{thm}

\begin{proof}
Let A be $\tau$-semi compact and $\{K_\lambda : \lambda \in \Lambda\}$ be a $\tau_N$ semi open cover of $A$. So $\exists$ open M-sets $S_\lambda \in \tau$ such that $C_{K_\lambda} (x) = C_{N \cap S_\lambda}(x) $ for each $\lambda,$ for all $x$.\\
Now $C_A (x) \leq C_N (x) \leq C_{\cup K_\lambda} (x) \leq C_{\cup S_\lambda}(x)$. Therefore $\{S_\lambda : \lambda \in \Lambda \}$ forms a $\tau$-semi open cover of $A$. So $\exists$ a finite sub cover $\{S_{\lambda_i} : i = 1, 2, . . . ,n\}$ such that $C_A (x) \leq C_N (x) \leq C_{\overset{n}{\underset{i=1}{\cup }} S_{\lambda_i}}(x)$. \\
But then $C_A (x) \leq C_N (x) \leq C_{N \cap (\overset{n}{\underset{i=1}{\cup }} S_{\lambda_i})}(x) = C_{ \overset{n}{\underset{i=1}{\cup }} (N \cap S_{\lambda_i})}(x) = C_{ \overset{n}{\underset{i=1}{\cup }} K_{\lambda_i}}(x) $.\\
This shows that $A$ is $\tau_N$ semi compact.\\
Conversely, let $\{S_\lambda\}$ be a $\tau$-semi open cover of $A$.\\ Setting  $C_{G_\lambda} (x) = C_{N \cap S_\lambda} (x)$ for all $x$ . \\
Then $C_A (x) \leq C_N (x) \leq C_{\cup S_\lambda} (x) \leq C_{N \cap (\cup S_\lambda)}(x)= C_{ \cup (N \cap S_\lambda)}(x) = C_{G_\lambda} (x) $.\\
So $\{G_\lambda : \lambda \in \Lambda \}$ is a $\tau_N$-semi open cover of $A$. By hypothesis, $\exists$ a finite subcollection $\{G_{\lambda_i} : i = 1, 2, . . . ,n\}$ such that $C_A (x) \leq  C_{\overset{n}{\underset{i=1}{\cup }} G_{\lambda_i}}(x) = C_{\overset{n}{\underset{i=1}{\cup }} (N \cap S_{\lambda_i})}(x) = C_{N \cap (\overset{n}{\underset{i=1}{\cup }}  S_{\lambda_i})}(x) \leq C_{\overset{n}{\underset{i=1}{\cup }}  S_{\lambda_i}}(x)$.\\
This implies $A$ is $\tau$-semi compact.
\end{proof}

\section{Conclusion}

In this paper, we have generalized open sets and closed sets in multiset topological space as SOM-sets and SCM-sets. With the help of which the notion of compactness is generalized as semi compactness in M-topology. The present article extend the theory of SOM-sets, SCM-sets and semi compactness of general sets to multisets. The study of various properties of these new sets shows that they deviate from classic theorems of topology in certain cases which are discussed in remarks. SOM-sets also have some extra properties which are listed in section 3. We also established a new characterization (Theorem 4.12 ) of semi compactness. \\
In future, one may consider generalization of separation axioms, countability axioms, connectedness and other topological properties in context of multiset theory.


\end{document}